%% file: carlson-simpson.tex
\documentclass[options]{amsart}

\usepackage{enumerate}
\usepackage{latexsym,ifthen,amssymb}
\usepackage{latexsym,ifthen,amssymb}
\usepackage[toc,page,title,titletoc,header]{appendix}
\usepackage{graphicx}
\usepackage{latexsym,ifthen,amssymb}
\usepackage[toc,page,title,titletoc,header]{appendix}
\usepackage{graphicx, amsmath, amsthm, amssymb,float,setspace,color, multirow}
\usepackage{tikz}
\usepackage{hyperref}
\usepackage{cleveref}

\hypersetup{
	pdfborder={0 0 0}
}
\usetikzlibrary{decorations.markings}

\usepackage{url}
\usepackage[margin=1in]{geometry}
\usepackage{xcolor}
\usepackage{soul}

\input{macros}

\DeclareMathOperator{\dom}{dom}
\newtheorem{theorem}{Theorem}[section]
\newtheorem{lemma}[theorem]{Lemma}

\newtheorem{corollary}[theorem]{Corollary}

\theoremstyle{definition}
\newtheorem{definition}[theorem]{Definition}

\newtheorem{example}[theorem]{Example}

\theoremstyle{remark}
\newtheorem{remark}[theorem]{Remark}
\newtheorem {question}[theorem]{Question}
\numberwithin{equation}{section}

\newtheoremstyle{noparens}%
  {}{}%
{}{}%
{\bfseries}{.}%
{ }%
{\thmname{#1}\thmnumber{ #2}\thmnote{ #3}}
\theoremstyle{noparens}
\newtheorem*{question*}{Question}
\newtheorem*{theorem*}{Theorem}

\newcommand{\nat}{\omega}

\newcommand{\ISig}{\mathsf{I}\Sigma^0}

\title[Carlson-Simpson's lemma and applications in reverse mathematics]{Carlson-Simpson's lemma and applications\\ in reverse mathematics}

\author{Paul-Elliot Angles d'Auriac}
\email{peada@free.fr}

\author{Lu Liu}
\address{School of Mathematics and Statistics, HNP-LAMA\\
Central South University\\
ChangSha 410083\\
People’s Republic of China}
\email{g.jiayi.liu@gmail.com}

\author{Bastien Mignoty}
\address{ENS Lyon\\46 all\'ee d'Italie\\69007 Lyon, FRANCE}
\email{bastien.mignoty@ens-lyon.fr}

\author{Ludovic Patey}
\address{CNRS, Équipe de Logique\\Universit\'e de Paris\\ Paris, FRANCE}
\email{ludovic.patey@computability.fr}

\def\msf{\mathsf}

\begin{document}

\maketitle

\begin{abstract}
We study the reverse mathematics of infinitary extensions of the Hales-Jewett theorem, due to Carlson and Simpson. These theorems have multiple applications in Ramsey's theory, such as the existence of finite big Ramsey numbers for the triangle-free graph, or the Dual Ramsey theorem. We show in particular that the Open Dual Ramsey theorem holds in $\mathsf{ACA}^{+}_0$.
\end{abstract}

\section{Introduction}\label[section]{sect_intro}
\input{parts/introduction.tex}

\section{Notation and background}\label[section]{sect_notation}
\input{parts/notation.tex}

\section{A proof of the Ordered Variable Word theorem in $\aca_0$}\label[section]{sect_ovw-comb-aca}
\input{parts/proof-ovw-aca.tex}

\section{A proof of the Density Ordered Variable Word theorem in $\aca_0$}\label[section]{sect_ovw-aca}
\input{parts/proof-ovw-density-aca.tex}

\section{A proof of the Higher-Order Carlson-Simpson Lemma in $\aca^{+}_0$}\label[section]{sect-ho-csl-acap}
\input{parts/proof-ho-carlson-simpson-acap.tex}

\section{A proof of the Open Dual Ramsey theorem in $\aca^{+}_0$}\label[section]{sect-cdrt-acap}
\input{parts/proof-cdrt-acap.tex}

\section{Big Ramsey number of the universal triangle-free graph using variable words}\label[section]{sect_henson-graph}
\input{parts/big-ramsey-triangle-free.tex}


%
%
%


\section*{Acknowledgement}

The authors are thankful to Natasha Dobrinen, Jan Hubička and Keita Yokoyama for interesting comments and discussions. The authors are also thankful for the anonymous referee for improvement suggestions.
Angles d'Auriac and Patey are partially supported by grant ANR ``ACTC'' \#ANR-19-CE48-0012-01. Lu Liu is supported by NSFC of Hunan Province with grant number 2022JJ30676.

\bibliographystyle{plain}
\bibliography{bibliography}

\end{document}

%% file: macros.tex
\newcommand{\Nb}{\mathbb{N}}

\newcommand{\Psf}{\mathsf{P}}

\newcommand{\Cc}{\mathcal{C}}

\newcommand{\Lc}{\mathcal{L}}

\newcommand{\Pc}{\mathcal{P}}

\renewcommand{\setminus}{\smallsetminus}

\newcommand{\dens}{\operatorname{dens}}

\newcommand{\card}[1]{\mathrm{card}\left(#1\right)}
\newcommand{\size}[1]{\mathrm{size}\left(#1\right)}

\newcommand{\s}[1]{\ensuremath{\sf{#1}}}

\DeclareMathOperator{\piooca}{\Pi^1_1\mbox{-}\s{CA}}
\DeclareMathOperator{\aca}{\s{ACA}}

\def\h{\hat}

\newcommand{\NN}{\mathbb{N}}
\newcommand{\ACA}{\aca}
\newcommand{\RCA}{\mathsf{RCA}}
\newcommand{\OVW}[2]{\mathsf{OVW}(#1,#2)}
\newcommand{\OVWD}[3]{\mathsf{OVW}^{#1}(#2,#3)}
\newcommand{\CSL}[2]{\mathsf{CSL}(#1,#2)}
\newcommand{\CSLD}[3]{\mathsf{CSL}^{#1}(#2,#3)}
\newcommand{\CDRT}[2]{\mathsf{CDRT}^{#1}(#2)}
\newcommand{\ODRT}[2]{\mathsf{ODRT}^{#1}(#2)}

\def\qt#1{``#1''}%

%% file: parts/introduction.tex
Tree partition theorems play an important role in Structural Ramsey theory, by providing a combinatorial core to which many other structural theorems can be reduced. For example, Milliken's tree theorem, states that, given an infinite finitely branching tree~$T$, for every finite coloring of the strongly embedded subtrees of length~$n$ of~$T$, there is a strongly embedded subtree of infinite height~$S$ such that all embedded subtrees of length~$n$ are monochromatic. Milliken's tree theorem is known to be the combinatorial core to prove the existence of a big Ramsey degree for colorings of the rationals and of the Rado graph (see Todorcevic~\cite{todorcevic2010introduction}).

We study the reverse mathematics of infinitary extensions of the Hales-Jewett theorem, due to Carlson and Simpson~\cite{Carlson1984dual}, and prove that they hold over~$\ACA_0$. The higher-order version of Carlson-Simpson's lemma holds in $\ACA^+_0$ and was used by Carlson and Simpson to prove a dual version of Ramsey's theorem~\cite{Carlson1984dual}. More recently, Hubi\v{c}ka~\cite{hubicka2020big} proved that the existence of a big Ramsey degree for the universal triangle-free graph followed from the higher-order version of Carlson-Simpson's lemma. This gave a simpler proof, yet less accurate, than the original proof of Dobrinen~\cite{dobrinen2020ramsey}. Thanks to our reverse mathematical analysis of Carlson-Simpson's lemma, we deduce that the existence of a big Ramsey degree for the universal triangle-free graph, and the restriction of the Dual Ramsey theorem to open sets both hold in~$\ACA^+_0$.


\subsection{Variable words}\label[section]{sect_intro-vw}

We identify a non-negative integer $k \in \omega$ with the set $\{0, \dots, k-1\}$.
A \emph{word} over a finite alphabet $A$ is a finite ordered sequence $w = \langle a_0, \dots, a_{n-1} \rangle \in A^n$ for some $n \in \omega$. An \emph{infinite word} over $A$ is a function $W : \omega \to A$. We denote by $A^{<\omega}$ and $A^\omega$ the sets of finite and infinite words over $A$, respectively. For $w = \langle a_0, \dots, a_{n-1} \rangle$, we write $|w|$ for the \emph{length} $n$ of the word~$w$ and given $i < n$, we let $w(i) = a_i$.

An \emph{$\omega$-variable word} over $A$ is an infinite word $W$ over the alphabet $A \sqcup \{ x_j : j \in \NN \}$ where each variable kind $x_j$ appears at least once and the first occurrence of $x_j$ appears before the first occurrence of $x_{j+1}$. We write $A^{\omega, \omega}$ for the set of all $\omega$-variable words over $A$.
Given an $\omega$-variable word $W$ over $A$ and a word $u$ over $A$, we write $W[u]$ for the finite word over $A$ where each occurrences of $x_j$ is replaced by $u(j)$, and cut before the first occurrence of $x_{|u|}$. In particular,  letting $\epsilon$ be the empty word, $W(\epsilon)$ is the initial segment of $w$ before the first occurence of $x_0$. The substitution notation $W[a]$ must not be confused with $W(i)$: the former notation denotes the finite word obtained by substitution of all the occurrences of~$x_0$ by the letter~$a$ and cutting before the first occurrence of~$x_1$, while the latter notation denotes the letter in $W$ at position~$i$.

\begin{example}
Then sequence $01101x_01010x_110x_0101x_20110x_001010x_1$ is a valid initial segment of an $\omega$-variable word over $2$. On the other hand, $010x_10101x_0$ is not, since the first occurrence of $x_1$ appears before the first occurrence of $x_0$. Similarly, $00191x_0101x_2$ is not a valid initial segment, since in an $\omega$-variable word over $2$, the variable $x_1$ must appear at some point, before the first occurrence of $x_2$.
If $W$ is an $\omega$-variable word over $2$ starting with $01x_010x_101x_0001x_2$, then $W[\epsilon] = 01$, $W[0] = 01010$, $W[1] = Ø110$,
$W[01] = 010101010001$ and $W[10] = 011100011001$.
\end{example}

The following theorem was proven by Carlson and Simpson~\cite[Lemma 2.4]{Carlson1984dual}:

\begin{theorem}[Carlson-Simpson Lemma\footnote{The name \qt{Carlson-Simpson Lemma} is sometimes used to refer to the Ordered Variable Word theorem in the literature of combinatorics.}]\label[theorem]{thm_csl}
For every finite alphabet $A$ and every finite partition $C_0 \sqcup \dots \sqcup C_{\ell-1} = A^{<\omega}$, there is some color $i < \ell$ and an $\omega$-variable word $W$ such that $\{ W[u] : u \in A^{<\omega} \} \subseteq C_i$.
\end{theorem}

We write $\msf{CSL}(k,\ell)$ the statement of \cref{thm_csl} for $\ell$-colorings and alphabets of size $k$.
The Carlson-Simpson Lemma has several consequences in combinatorics, among which a dual version of Ramsey's theorem and the existence of a big Ramsey degree of the universal triangle-free graph.

Carlson and Simpson~\cite[Theorem 6.3]{Carlson1984dual} actually proved a stronger statement, known as the Ordered Variable Word theorem. An $\omega$-variable word over $A$ is \emph{ordered} if the last occurrence of $x_j$ appears before the first occurrence of $x_{j+1}$. Note that if $W$ is an ordered $\omega$-variable word, then each variable must appear finitely often, unlike in the general case. We write $A^{\omega, \omega}_{<}$ for the set of all ordered $\omega$-variable words over $A$.

\begin{theorem}[Ordered Variable Word theorem]\label[theorem]{thm_ovw}
For every finite alphabet $A$ and every finite partition $C_0 \sqcup \dots \sqcup C_{\ell-1} = A^{<\omega}$,
there is some color $i < \ell$, and an ordered $\omega$-variable word $W$ over $A$ such that $\{ W[u] : u \in A^{<\omega} \} \subseteq C_i$.
\end{theorem}

We write $\OVW k \ell$ the statement of \cref{thm_ovw} for $\ell$-colorings of words over finite alphabets of size $k$. We study the reverse mathematics of the Ordered Variable Word theorem in \Cref{sect_ovw-comb-aca} and \Cref{sect_ovw-aca}, and prove that it holds over $\ACA_0$. It is currently unknown whether $\OVW k \ell$ is strictly weaker.

\subsection{Higher-order Variable Words}\label[section]{sect_intro-ho-vw}

The same way Ramsey's theorem for $n$-tuples can be proven inductively from the pigeonhole principle, the Carlson-Simpson Lemma can be used as a pigeonhole principle to prove inductively a higher-order version coloring finite multivariable  words. An (ordered) \emph{$n$-variable word} over $A$ is a word $w$ over the alphabet $A \sqcup \{ x_j : j < n \}$ where each $x_j$ appears at least once and the first (last) occurrence of $x_j$ appears before the first occurrence of $x_{j+1}$. We call $n$ the \emph{dimension} of the $n$-variable word $w$.
Denote by $A^{<\omega, n}$ and $A^{<\omega, n}_{<}$ the sets of unordered and ordered $n$-variable words over $A$, respectively. Note that $A^{<\omega} = A^{<\omega, 0}$, and that $A^{<\omega, n} \subseteq (A \sqcup \{x_0, \dots, x_{n-1}\})^{<\omega}$.
The higher-order version of the Carlson-Simpson Lemma is about finite colorings of $n$-variable words. The \emph{order} of the theorem is the dimension of the variable words which are colored.

\begin{theorem}[Higher-order Carlson-Simpson Lemma]\label[theorem]{thm_ho-csl}
Fix $n \geq 0$ and $\ell \geq 1$.
For every finite alphabet $A$ and every finite partition $C_0 \sqcup \dots \sqcup C_{\ell-1} = A^{<\omega, n}$, 	
there is some color $i < \ell$ and an infinite $\omega$-variable word $W$ such that $\{W[u] : u \in A^{<\omega, n}\} \subseteq C_i$.
\end{theorem}

We write $\CSLD n k \ell$ the statement of \cref{thm_ho-csl} for $\ell$-colorings of $n$-variable words over finite alphabets of size $k$. In particular, $\CSLD 0 k \ell$ is the statement $\CSL k \ell$.
We study the reverse mathematics of the Higher-order Carlson-Simpson Lemma in \Cref{sect-ho-csl-acap} and prove that it holds over $\ACA_0^{+}$.

The Ordered Variable Word theorem also admits a higher-order counterpart, due to Carlson.
The two last theorems of this introduction will not be studied in this paper, but we state them for the sake of completion:

\begin{theorem}[Higher-order Ordered Variable Word theorem]\label[theorem]{thm_ho-ovw}
Fix $n \geq 0$ and $\ell \geq 1$.
For every finite alphabet $A$ and every finite partition $C_0 \sqcup \dots \sqcup C_{\ell-1} = A^{<\omega, n}_{<}$, 	
there is some color $i < \ell$ and an infinite ordered $\omega$-variable word $W$ such that $\{W[u] : u \in A^{<\omega, n}\} \subseteq C_i$.
\end{theorem}

We write $\OVWD n k \ell$ the statement of \cref{thm_ho-ovw} for $\ell$-colorings of ordered $n$-variable words over finite alphabets of size $k$. Here again, $\OVWD 0 k \ell$ is the statement $\OVW k \ell$.
However, $\OVWD n k \ell$  cannot be proven by iterating its zero-dimensional version as in the case of the Carlson-Simpson Lemma. \Cref{thm_ho-ovw} follows from a stronger theorem known as Carlson's theorem. $\OVWD n k \ell$ and Carlson's theorem have not been studied so far in reverse mathematics, as the only two known proofs of Carlson's theorem involve third-order objects, namely, ultrafilters in combinatorics, and Ellis enveloping semigroup in topological dynamics.

The Higher-order Ordered Variable Word theorem of dimension 1 for unary alphabets is actually equivalent to Hindman's theorem over~$\RCA_0$. Given a set $X \subseteq \omega$, we write $\mathsf{FS}(X)$ for the set of all non-empty sums over $X$ with no repetitions, that is, $\mathsf{FS}(X) = \{  \sum F : F \subseteq_{\mathtt{fin}} X \wedge F \neq \emptyset \}$.

\begin{theorem}[Hindman's theorem]
For every finite partition $C_0 \sqcup \dots \sqcup C_{\ell-1} = \omega$;
there is some color $i < \ell$ and an infinite set $X \subseteq \omega$ such that $\mathsf{FS}(X) \subseteq C_i$.
\end{theorem}

There exist several proofs of Hindman's theorem, which was extensively studied in reverse mathematics. Hindman's theorem is provable in $\ACA_0^{+}$ and implies $\ACA_0$. The exact strength of Hindman's theorem is one of the most important open question in reverse mathematics.

\subsection{Organization of the paper}

In \Cref{sect_notation}, we fix some notation and definitions which will be useful all along the paper.
Then, in \Cref{sect_ovw-comb-aca} and \Cref{sect_ovw-aca}, we give two proofs of the Ordered Variable Word theorem in~$\ACA_0$, based on two different largeness notions: piecewise syndeticity and positive upper density. In \Cref{sect-ho-csl-acap}, we iterate either proof of the Ordered Variable Word in~$\ACA_0$ to obtain a proof of the Higher-order Carlson-Simpson Lemma in~$\ACA^+_0$. Last, we explore two applications of the Higher-order Carlson-Simpson Lemma, namely, the Dual Ramsey Theorem for open sets in \Cref{sect-cdrt-acap} and the existence of a big Ramsey degree of the universal triangle-free graph in \Cref{sect_henson-graph}. Both consequences are shown to hold over~$\ACA^+_0$.


%% file: parts/notation.tex


In the introduction, we stated the Ordered Variable Word theorem and the Carlson-Simpson Lemma in terms of variable words. However, it is sometimes more convenient 
to consider the set of words obtained by taking all the possible instantiations of a variable word.

\begin{definition}[Ordered Variable Word tree]
  An \emph{OVW-tree} over $A$ of dimension $n \in \omega \cup \{\omega\}$
  is a set of the form $T = \{ w[u] : u \in A^{<n} \}$
  for some ordered $n$-variable word $w$ over $A$.
  We call $w$ its \emph{generating variable word}.
  An \emph{OVW-line} is an OVW-tree of dimension 1.
\end{definition}

We write $T(j)=\{w[u] : u \in A^{j} \}$ for the \emph{$j$-th level of $T$}, $\Lc(T) = \{|u|:u\in T\}$ for the \emph{set of levels of $T$} and $|T|=\max(\Lc(T))$ for the \emph{size of $T$}. Note that the size of $T$ is different from the cardinality of $T$ as a set, and from the dimension of $T$ as well. The size of $T$ coincides with the length of its generating variable word. An \emph{OVW-subtree} of $T$ is an OVW-tree which is a subset of $T$. 

\begin{example}
For any $c \in A^{<\omega}$, $\{c\}$ is an OVW-tree of dimension $0$.
The OVW-trees of dimension 1 are the sets of the form $T = \{c\} \cup \{ c^\frown a^\frown w[a] : a \in A \}$ for some variable word $w$ over $A$. Say $A = \{0,1\}$, $c = 10$ and $w = 01x_0 10$. Then $T = \{10, 10001010, 10101110\}$, $T(0) = \{10\}$ and $T(1) = \{10001010, 10101110\}$. 
On the other hand, $S = \{10, 1010, 1001\}$ is not an OVW-tree.
\end{example}

It is easy to see that there is a one-to-one correspondence between OVW-trees and their generating variable words.
The tree presentation is especially convenient when dealing with iterations.
The Ordered Variable Word theorem can be stated in terms of OVW-trees as follows:

\begin{theorem}[Ordered Variable Word theorem]\label[theorem]{thm_ovw-tree-iterable}
For every finite alphabet $A$, every OVW-tree $T \subseteq A^{<\omega}$ over $A$ of dimension $\omega$ and every finite partition $C_0 \sqcup \dots \sqcup C_{\ell-1} = T$,
there is some color $i < \ell$ and an OVW-subtree $S \subseteq T$ of dimension $\omega$ such that  $S \subseteq C_i$.
\end{theorem}
\begin{proof}
Let $f : T \to A^{<\omega}$ be the canonical computable isomorphism. Define $D_0 \sqcup \dots \sqcup D_{\ell-1} = A^{<\omega}$ by $D_i = \{ u \in A^{<\omega} : f(u) \in C_i \}$.
By \cref{thm_ovw}, there is an ordered $\omega$-variable word $W$ over $A$ and a color $i < \ell$ such that $\{ W[u] : u \in A^{<\omega} \} \subseteq D_i$. In particular, $S = \{ f^{-1}(W[u]) : u \in A^{<\omega} \}$ is an OVW-subtree of $T$ such that $S \subseteq C_i$.
\end{proof}

In some occasions, it will also be convenient to see an ordered $n$-variable word as a finite sequence $\sigma, w_0, w_1, \dots, w_n$ where~$\sigma \in A^{<\omega}$ is a word over $A$, and $w_i$ are \emph{left 1-variable words} over~$A$, that is, 1-variable words such that the variable occurs first at position~0. There is again a one-to-one correspondence between OVW-trees of dimension~$n$ and sequences of this form.

\subsection{Largeness and partition regularity}

Partition theorems are often refined in terms of large sets which are partition regular. 
These refinements can be seen as quantitative versions of these theorems. As it happens, the proofs of the refined versions are sometimes more elementary from a reverse mathematical viewpoint, although combinatorially more complicated.

\begin{definition}
A class~$\Cc \subseteq \Pc(A^{<\nat})$ is \emph{partition regular} if
\begin{enumerate}
	\item[(1)] it is non-empty
	\item[(2)] if $B \in \Cc$ and $B \subseteq C$, then $C \in \Cc$
	\item[(3)] if $B \in \Cc$ and $C_0 \sqcup C_1 = B$, then either~$C_0 \in \Cc$ or $C_1 \in Cc$
\end{enumerate}
\end{definition}

We shall consider two refinements of the Ordered Variable Words, based on two standard partition regular notions : positive upper density and piecewise syndeticity.

\subsubsection{Positive upper density}

\begin{definition}
We define the density of a set $D$ inside a set $U$ as $\dens_U(D) = \frac{\card {D\cap U}}{\card U}$.
A set~$D \subseteq A^{<\nat}$ has \emph{positive upper density} if 
$$
\limsup_{n \to \infty} \dens_{A^n}(D) > 0
$$
\end{definition}

It is clear that $A^{<\nat}$ has positive upper density, and that if~$D$ has positive upper density, then so have its supersets. Thus, properties (1) and (2) of the definition of partition regularity are satisfied for the class of all sets of positive upper density. The following lemma shows that property (3) is also satisfied, hence the class is partition regular.

\begin{lemma}[$\RCA_0$]
Suppose~$D \subseteq A^{<\nat}$ has positive upper density, and $E \sqcup F = D$.
Then either~$E$ or $F$ has positive upper density.
\end{lemma}
\begin{proof}
Let~$\epsilon > 0$ be such that the set $B = \{ n : \dens_{A^n}(D) > \epsilon \}$ is infinite.
Suppose that~$E$ does not have positive upper density. Then the set~$C = \{ n : \dens_{A^n}(E) > \epsilon/2 \}$ is finite.
It follows that the set~$B \setminus C$ is infinite.
Note that $\forall n \in B \setminus C$, $\dens_{A^n}(D) = \dens_{A^n}(E) + \dens_{A^n}(F) > \epsilon$ and $\dens_{A^n}(E) \leq \epsilon/2$, so $\dens_{A^n}(F) > \epsilon/2$.
It follows that $\limsup_{n \to \infty}  \dens_{A^n}(F) > \epsilon/2$.
\end{proof}

We will prove in \Cref{sect_ovw-aca} that every set of positive upper density admits a solution to the Ordered Variable Word theorem.

\subsubsection{Thickness and syndeticity}
Given a set~$F \subseteq A^{<\nat}$ and a word~$\sigma \in A^{<\nat}$,
we let $F \cdot \sigma = \{ \tau\sigma : \tau \in F \}$.

\begin{definition}
A set~$S \subseteq A^{<\nat}$ is \emph{syndetic} if there is some~$\ell$ such that for every~$\sigma \in A^{<\nat}$, there is some~$\tau \in A^{\leq \ell}$ such that $\tau\sigma \in S$. For a given $\ell$ we call such a set $\ell$-syndetic.
A set~$T \subseteq A^{<\nat}$ is \emph{thick} if 
for each~$\ell \in \nat$, there is some~$\sigma \in A^{<\nat}$ such that $A^{\leq \ell} \cdot \sigma \subseteq T$.
A set~$P \subseteq A^{<\nat}$ is \emph{piecewise syndetic} if it is the intersection of a thick set and a syndetic set.
\end{definition}

Thickness and syndeticity are not partition regular notions, as they are not closed under partitioning. They play a dual role, in that a set is thick if and only if it intersects every syndetic set. We shall see on the other hand that piecewise syndeticity is partition regular. 

\begin{lemma}[$\RCA_0$]\label[lemma]{lem:brown-lemma}
Let~$P \subseteq A^{<\nat}$ be a piecewise syndetic set, and let~$P = B \sqcup C$, then either $B$ or $C$ is piecewise syndetic.
\end{lemma}
\begin{proof}
Assume that $P=S\cap T$, where $S$ is syndetic and $T$ is thick, and let $P = B \sqcup C$. Let $\tilde S:=B\cup(S\setminus P)$. First, note that $B \subseteq \tilde S$ and $B \subseteq P \subseteq T$, so $B \subseteq \tilde S \cap T$. Also note that $\tilde S \cap T \subseteq B \cup ((S \setminus P) \cap T) \subseteq B$, so $B=\tilde S\cap T$. It follows that if $\tilde S$ is syndetic, then $B$ is piecewise syndetic, and the proof is finished. But if, on the contrary, $\tilde S$ is not syndetic, then $\tilde T:=A^{<\omega}\setminus\tilde S$ is thick. Finally observe that $C= P \setminus B = (S \setminus B) \cap (S \cap P) = (S \setminus B) \cap (S \setminus (S \setminus P)) = S \setminus \tilde{S} = (A^{<\omega} \setminus \tilde{S}) \cap S = \tilde T\cap S$, which shows that in this case $C$ is piecewise syndetic.
\end{proof}

We will actually need the following iterated version of the previous lemma, which is a generalization of the so-called Brown's lemma. Brown's lemma was originally proved for locally finite semigroups. Frittaion~\cite{frittaion2017browns} studied it from a reverse mathematical viewpoint for the semigroup~$(\mathbb{N}, +)$ and showed that it is equivalent to $\ISig_2$ over~$\RCA_0$. The following proof is essentially the same, recasted in the setting of the semigroup $(A^{<\omega}, \cdot)$.

\begin{lemma}[Brown, $\RCA_0 + \ISig_2$]\label[lemma]{cor:brown-lemma}
Let~$P \subseteq A^{<\nat}$ be a piecewise syndetic set, and let~$P = \sqcup_{i < k} C_i$ for some~$k \in \omega$.
Then there is some~$i < k$ such that $C_i$ is piecewise syndetic.
\end{lemma}
\begin{proof}
Let~$P = S \cap T$, where~$S$ is syndetic and $T$ is thick. 
By bounded $\Sigma^0_2$ comprehension, the following set exists:
$$
I = \{ B \subseteq \{0, \dots, k-1\} : (A^{<\omega} \setminus T) \cup \bigcup_{i \in B} C_i \mbox{ is syndetic } \}
$$
Note that~$P \cup (A^{<\omega} \setminus T)$ is syndetic, hence $\{0, \dots, k-1\} \in I$.
Let~$B \in I$ be minimal for the inclusion. Note that~$B \neq \emptyset$. Fix any $i \in B$.
Since~$(A^{<\omega} \setminus T) \cup \bigcup_{j \in B} C_j \}$ is syndetic, then either
$(A^{<\omega} \setminus T) \cup \bigcup_{j \in B \setminus \{i\}} C_j$ is syndetic, or $C_i$ is piecewise syndetic. The former case would contradict minimality of~$B$, so the latter case holds.
\end{proof}

Before finishing this section, we prove a small technical lemma which will be useful in the proof of the Ordered Variable Word theorem.

%
%

\begin{lemma}[$\RCA_0$]\label[lemma]{lem:thick-subset-thick}
Suppose $I \subseteq A^{<\nat}$ is a thick set and $\ell \in \nat$. 
Then the set $J = \{ \sigma \in I : A^{\leq \ell} \cdot \sigma \subseteq I \}$ is thick.
\end{lemma}
\begin{proof}
Fix some~$m \in \nat$. Since $I$ is thick, there is some~$\tau \in A^{<\nat}$ such that
$A^{\leq m+\ell} \cdot \tau \subseteq I$. Let us show that $A^{\leq m} \cdot \tau \subseteq J$, in other words, for every~$\sigma \in A^{\leq m}$, $A^{\leq \ell} \cdot \sigma\tau \in I$.
Fix any~$\sigma \in A^{\leq m}$ and $\rho \in A^{\leq \ell}$. Then since $A^{\leq m+\ell} \cdot \tau \subseteq I$, $\rho\sigma\tau \in I$.
\end{proof}

%% file: parts/proof-ovw-aca.tex
The purpose of this section is to prove the following piecewise syndetic version of the Ordered Variable Word over~$\ACA_0$.

\begin{theorem}[Piecewise Syndetic Ordered Variable Word Theorem, $\ACA_0$]\label[theorem]{thm:piecewise-syndetic-ovw}
Let~$P \subseteq A^{<\nat}$ be a piecewise syndetic set.
Then there exists an OVW-tree $T \subseteq P$ over~$A$ of dimension~$\omega$.
\end{theorem}

Since piecewise syndeticity is partition regular, \Cref{thm:piecewise-syndetic-ovw} implies in particular the Ordered Variable Word theorem.

\begin{corollary}[$\ACA_0$]
For every finite alphabet~$A$ and every finite partition~$C_0 \sqcup \dots \sqcup C_{\ell-1} = A^{<\nat}$,
there is some color~$i < \ell$ and an OVW-tree $T \subseteq C_i$ of dimension~$\omega$.
\end{corollary}
\begin{proof}
By Brown's lemma (see \Cref{cor:brown-lemma}), there is some~$i < \ell$ such that $C_i$ is piecewise syndetic. By \Cref{thm:piecewise-syndetic-ovw}, there exists an OVW-tree $T \subseteq C_i$ over~$A$ of dimension~$\omega$.
\end{proof}



We are going to use the following finitary version of the Ordered Variable Word. 

\begin{theorem}[$\RCA_0$]\label[theorem]{cor:finite-ovw-letter}
Fix a finite alphabet~$A$, a finite set of colors~$C$.
For every coloring $f : A^{<\omega} \to C$, there is a monochromatic OVW-line~$S \subseteq A^{<\omega}$ and a letter~$a \in A$ such that $S(0)$ and $S(1)\cdot a$ are both $f$-homogeneous for the same color.
\end{theorem}

It follows by compactness from the Ordered Variable Word theorem, but Dodos, Kanellopoulos and Tyros~\cite[Section 4]{Dodos2014DensityCS} gave an elementary proof which can be formalized in~$\RCA_0$. 

\begin{remark}
The statement of \Cref{cor:finite-ovw-letter} is slightly different from the one of~\cite[Section 4]{Dodos2014DensityCS}, but can be recovered by taking a monochromatic  OVW-tree $T \subseteq A^{<\omega}$ of dimension~2, then picking any non-empty  $\sigma \in A^{<\omega}$ such that $T(1)\cdot \sigma \subseteq T(2)$. Let~$\sigma^{*}$ be the word~$\sigma$ truncated from its last letter~$a$,
and let~$S = T(0) \cup   (T(1)\cdot\sigma^{*}) $. Then $S$ is an OVW-line such that $S(0)$ and $S(1)\cdot a$ are both $f$-homogeneous for the same color.
\end{remark}

We are now ready to prove the main combinatorial lemma. The piecewise syndetic version of the Ordered Variable Word follows by iterating the following lemma.

\begin{lemma}[$\RCA_0 + \ISig_2$]\label[lemma]{lem:coloring-thick-left}
Let~$I \subseteq A^{<\omega}$ be a thick set and $X : I \to C$ be a coloring.
There is an OVW-line $S \subseteq A^{<\omega}$ and a piecewise syndetic $P \subseteq I$ such that 
 $S(0)$ and $S(1) \cdot P$ are $X$-homogeneous for the same color.
\end{lemma}
\begin{proof}
Fix~$I$ and $C$. Suppose the lemma does not hold. We will build an infinite sequence of words~$\sigma_0, \sigma_1, \dots$ such that, letting $w_n$ be the left variable word~$x_0 \sigma_n$ and $N_n = \sum_{m \leq n} |w_n|$, the following property holds (for convenience, let~$N_{-1} = 0$):
\begin{equation}\label{eq-thick-neg2}
	\text{\parbox{.81\textwidth}{For every~$n \in \omega$, every~$b \in A$, every OVW-line $S \subseteq A^{<\omega}$ such that $|S| = N_{n-1}$, with $S(0) \subseteq I$, we have $A^{N_{n-1}+1} \cdot \sigma_n \subseteq I$ and $S(0)$ and $S(1)\cdot w_n[b]$ are not $X$-homogeneous for the same color.}}
\end{equation}
Let~$\sigma_0$ be such that $A \cdot \sigma_0 \subseteq I$. Assume $\sigma_0, \dots, \sigma_{n-1}$ have been defined.

Let~$K$ be the set of all~$\tau \in I$ such that there is an OVW-line $S \subseteq A^{<\omega}$ with $|S| = N_{n-1}$ and $S(0) \subseteq I$ such that $S(0)$ and $S(1)\cdot \tau$ are $X$-homogeneous for the same color. If~$K$ is piecewise syndetic, then by Brown's lemma (\Cref{cor:brown-lemma}), there is an OVW-line $S \subseteq A^{<\omega}$ such that $|S| = N_{n-1}$, with $S(0) \subseteq I$, and a piecewise syndetic set~$P \subseteq K$ such that
$S(0)$ and $S(1) \cdot P$ are $X$-homogeneous for the same color. Then the lemma is satisfied and we are done.

Otherwise, the set $J = I \setminus K$ is thick. Then for every~
OVW-line $S \subseteq A^{<\omega}$ with $|S| = N_{n-1}$ and $S(0) \subseteq I$,   every~$\tau \in J$ with
 $S(1)\cdot \tau \subseteq I$,
 we have $S(0)$ and $S(1)\cdot \tau$ are not $X$-homogeneous for the same color.
Let~$\sigma_n$ be such that $A^{\leq N_{n-1}+1} \cdot \sigma_n \subseteq J$, and let~$w_n = x_0 \sigma_n$.
We claim that~$\sigma_n$ satisfies (\ref{eq-thick-neg2}). Fix any~$b \in A$, any OVW-line $S \subseteq A^{<\omega}$ with $|S| = N_{n-1}$ and $S(0) \subseteq I$. Clearly, $w_n[b] = b \sigma_n \in J$
and $S(1)\cdot w_n[b]  \subseteq I$ (since $A^{\leq N_{n-1}+1} \cdot \sigma_n \subseteq J$).
Therefore, $S(0)$ and $S(1)\cdot w_n[b]$ are not $X$-homogeneous for the same color.
 Thus~$\sigma_n$ satisfies~(\ref{eq-thick-neg2}).

Consider the embedding $h : A^{<\omega} \to C$ defined by $h(a_0\cdots a_n) = w_0[a_0] \cdots w_n[a_n]$,
and let~$Y = X \circ h$. Note that by choice of~$(\sigma_n)_{n \in \omega}$, $\dom h = A^{<\omega}$.
By \Cref{cor:finite-ovw-letter}, there is an OVW-line $T \subseteq A^{<\omega}$, a letter $b \in A$, and a color~$i \in C$ such that $T(0)$ and $T(1)\cdot b$ are both $Y$-homogeneous for color~$i$.
Let~$S \subseteq J$ be the OVW-line obtained by taking the image of~$T$ by $h$.
In particular, $S(0) = h(T(0))$ and $S(1)\cdot w_M[b] = h(T(1)\cdot b)$ and $|S| = N_{M-1}$, where $M = |T|+1$.
By definition of~$h$ and $Y$, $S(0)$ and $S(1)\cdot w_M[b]$ are both~$X$-homogeneous for color~$i$. This contradicts (\ref{eq-thick-neg2}).
\end{proof}

It will be convenient to reformulate the previous lemma into the following equivalent lemma, which is in terms of piecewise syndetic sets instead of finite colorings of thick sets. One can indeed see a piecewise syndetic set as a particular coloring of a thick set, where the color is the witness of syndeticity.

\begin{lemma}[$\RCA_0 + \ISig_2$]\label[lemma]{lem:coloring-thick-left-ps}
Let~$P \subseteq A^{<\nat}$ be a piecewise syndetic set.
Then there is an OVW-line $S \subseteq A^{<\nat}$ and a piecewise syndetic subset~$Q $ such that $S(0) \subseteq P$ and  $S(1) \cdot Q \subseteq P$.
\end{lemma}
\begin{proof}
Say~$P = \h P \cap I$, where $\h P$ is $m$-syndetic and $I$ is thick.
By \Cref{lem:thick-subset-thick}, the set $J = \{ \sigma \in I : A^{\leq m} \cdot \sigma \subseteq I \}$ is thick.
Let~$X : J \to A^{\leq m}$ be defined by $X(\sigma) = \rho$ such that $\rho\sigma \in \h P$. Note that
\begin{align}\label{cseq1}
\text{ $X(\sigma)\cdot\sigma \in P$ for all $\sigma\in J$.}
\end{align}
By \Cref{lem:coloring-thick-left}, there is an OVW-line $T \subseteq A^{<\nat}$, a piecewise syndetic subset~$Q \subseteq J$ and a color~$\rho$ such that
$T(0)$ and $T(1)\cdot Q$ are in color $\rho $ of $X$.
By (\ref{cseq1}), $\rho\cdot T(0), \rho\cdot T(1)\cdot Q\subseteq  P$.
Thus the OVW-line $S = \rho\cdot T$ is as desired.
\end{proof}

We are now ready to prove \Cref{thm:piecewise-syndetic-ovw}.

\begin{proof}[Proof of \Cref{thm:piecewise-syndetic-ovw}]
Let $P \subseteq A^{<\nat}$ be a piecewise syndetic set.
By \Cref{lem:coloring-thick-left-ps}, there an OVW-line~$S$ and a piecewise syndetic set~$P_0 \subseteq A^{<\nat}$ such that $S(0) \subseteq P$ and   $S(1) \cdot  P_0 \subseteq P$. Set~$T_0 = S(0)$ and let $w_0$ be the left variable word such that $S(1) = S(0) \cdot   w_0[A] $. Note that $T_0$ is an OVW-tree of dimension~0.

Assume  by induction $T_s$ is an OVW-tree over~$A$ of dimension~$s$, $w_s$ is a left variable word and $P_s \subseteq A^{<\nat}$ is a piecewise syndetic set such that
\begin{enumerate}[(a)]
\item  $T_s \subseteq P$ and
\item   $T_s(s) \cdot w_s[A] \cdot P_s \subseteq P.$
\end{enumerate}
By \Cref{lem:coloring-thick-left-ps}, there is an OVW-line~$S$ and a piecewise syndetic set~$P_{s+1} \subseteq A^{<\nat}$ such that  $S(0) \subseteq P_s$ and
$S(1) \cdot  P_{s+1} \subseteq P_s$.
Let~$$T_{s+1} = T_s \cup
(T_s(s)\cdot w_s[A]\cdot S(0)) $$ and let $w_{s+1}$ be the left variable word such that $S(1) = S(0) \cdot   w_{s+1}[A]$.
Note that $T_{s+1}$ is an OVW-tree over~$A$ of dimension~$s+1$. It suffices to verify the induction assumption.

\emph{Claim 1:} $T_{s+1} \subseteq P$. Indeed, by (a) $T_s \subseteq P$ and by (b) and   $S(0) \subseteq P_s$, $$T_{s+1}(s+1) = T_s(s)\cdot w_s[A]\cdot S(0) \subseteq
T_s(s)\cdot w_s[A]\cdot P_s\subseteq P.$$

\emph{Claim 2:}   $T_{s+1}(s+1) \cdot w_{s+1}[A] \cdot P_{s+1} \subseteq P$.
\begin{align}\nonumber
T_{s+1}(s+1) \cdot w_{s+1}[A] \cdot P_{s+1}
 &= T_s(s)\cdot w_s[A]\cdot S(0)\cdot w_{s+1}[A]\cdot P_{s+1}
\\ \nonumber
&= T_s(s)\cdot w_s[A]\cdot S(1)\cdot P_{s+1}
\\ \nonumber
&\subseteq T_s(s)\cdot w_s[A]\cdot P_s
\subseteq P.
\end{align}
The four equalities are due to: definition of $T_{s+1}$,
definition of $w_{s+1}$, $S(1)\cdot P_{s+1}\subseteq P_s$
and (b).
\end{proof}

\begin{remark}\label[remark]{rem:full-induction-aca}
Although the proofs of \Cref{lem:coloring-thick-left-ps} and \Cref{lem:coloring-thick-left} are over~$\RCA$, that is, $\RCA_0$ with more induction, the proof of \Cref{thm:piecewise-syndetic-ovw} holds in $\ACA_0$ with restricted induction. Indeed, $\ACA_0$ proves the existence of a countable coded $\omega$-model of~$\RCA_0$ (see Theorem VIII.2.11 of Simpson~\cite{Simpson2009Subsystems}). Since every countable coded  $\omega$-model satisfies full induction, then every $\Pi^1_2$ consequence of $\RCA$ is provable in~$\ACA_0$. Simply note that \Cref{lem:coloring-thick-left-ps} and \Cref{lem:coloring-thick-left} are $\Pi^1_2$ statements.
\end{remark}

Note that if the coloring is computable, then the sequences built in the proof of \Cref{thm:piecewise-syndetic-ovw} are computable in any PA over $\emptyset'$. Therefore, for any set $P$ of PA degree over $\emptyset'$, any computable instance of the Ordered Variable Word theorem admits a $P$-computable solution.
We say that a problem $\Psf$ admits \emph{cone avoidance} if for every pair of sets $Z, C$ such that $C \not \leq_C Z$, every $Z$-computable instance of $\Psf$ admits a solution $Y$ such that $C \not \leq_T Z \oplus Y$. It admits \emph{strong cone avoidance} if we relax the requirement on $Z$ to be computable in the definition of cone avoidance.
The following question remains open:

\begin{question}
Does the Ordered Variable Word theorem admit cone avoidance
or even strong cone avoidance?
\end{question}

%% file: parts/proof-ovw-density-aca.tex
Dodos, Kanellopoulos and Tyros proved in \cite{Dodos2014DensityCS} a density version of the Ordered Variable Word theorem, a stronger result where the color is fixed and of positive $\limsup$. The goal of this section is to show that their result holds in $\ACA_0$:

\begin{theorem}[Density Ordered Variable Word Theorem, \cite{Dodos2014DensityCS}, $\ACA_0$] \label[theorem]{th:density-OVW}
Fix a finite alphabet $A$, and let $D \subseteq A^{<\omega}$ be of positive upper density. Then there exists an OVW-tree $T \subseteq D$ over $A$ of dimension~$\omega$.
  \end{theorem}

As for piecewise syndeticity, positive upper density is partition regular, so \Cref{th:density-OVW} implies the Ordered Variable Word theorem.

\begin{corollary}[$\ACA_0$]
For every finite alphabet~$A$ and every finite partition~$C_0 \sqcup \dots \sqcup C_{\ell-1} = A^{<\nat}$,
there is some color~$i < \ell$ and an OVW-tree $T \subseteq C_i$ of dimension~$\omega$.
\end{corollary}
\begin{proof}
Let~$i < \ell$ be such that $\limsup_ r\dens_{A^ r}(C_i)>0$. By \Cref{th:density-OVW},
there is an  OVW-tree $T \subseteq C_i$ over $A$ of dimension~$\omega$.
\end{proof}

The OVW-tree $T$ of dimension $\omega$ of \cref{th:density-OVW} will be constructed as the
union
of a sequence of OVW-tree of finite but increasing dimension.
The ``finitary part'' of the proof of \cref{th:density-OVW} is the following theorem, which is a direct consequence of {\cite[Proposition 7.5]{Dodos2014DensityCS}}.
For a set $D\subseteq A^{<\omega}$,
a finite OVW-line $S$, let
$D^S=\{\sigma: S(1)\cdot\sigma\subseteq D\}$.
\begin{theorem}[{\cite[Proposition 7.5]{Dodos2014DensityCS}}, $\RCA$]\label[theorem]{th:finitary-part-density-OVW}
  For all $k\in\Nb$ and alphabet $A$ of size $k$, for any $\delta>0$,  there exists $N_0=N_0(k,\delta)$
 such that for any $n\in\omega$, there exists $N_1=N_1(k,n,\delta)$ such that
   the following is true.
   For any $L_0<L_1\subseteq\Nb$ with $|L_0|>N_0$ and $|L_1|>N_1$, and for any $D$ with
    $\dens_{A^ r}(D)>\delta$ for all
     $  r\in  L=L_0\cup L_1$,  there exists an OVW-line $S\subseteq D$ with $\Lc(S)\subseteq L_0$, and a set $L'\subseteq L_1 $ such that:
  \begin{enumerate}
  \item $L'$ is sufficiently big:  $\card {L'}\geq n$, and
  \item $D^S$ is of sufficiently big density, at the length inside $L'$: for any $  r\in L'$, $$\dens_{A^ {r-\size S}}(D^S)>\frac{\delta^2}{8\times\card{\mathrm{OVWLine}(L_0)}}$$ where $\mathrm{OVWLine}(L_0)$ is the finite set of OVW-lines with set of levels included in $L_0$.
  \end{enumerate}
\end{theorem}

In order to study the reverse mathematics of \Cref{th:density-OVW}, we shall consider \Cref{th:finitary-part-density-OVW} as a blackbox.
Its proof involves an elaborate, but finite, combinatorial machinery, which is elementary from a logical viewpoint.
\Cref{th:finitary-part-density-OVW} for~$n = 1$ implies the following lemma, which is the analog of \Cref{lem:coloring-thick-left-ps}, with piecewise syndetic largeness replaced by positive upper density.


\begin{lemma}[$\RCA$]\label[lemma]{lem_technical-density}
Let $k$ and alphabet $A$ of size $k$. For every positive upper density set $D\subseteq A^{<\omega}$,
  there is an OVW-line~$S$, a positive upper density set~$\h D \subseteq A^{<\nat}$
 such that
$S(0) \subseteq D$ and
$S(1) \cdot  \h D \subseteq D$.
\end{lemma}
\begin{proof}
Let~$\delta > 0$ be such that the set $L_1 = \{ r \in \omega : \dens_{A^r}(D) > \delta \}$ is is infinite.
Let~$N_0 = N_0(k, \delta)$ and $N_1 = N_1(k, 1, \delta)$, and let~$\hat{\delta} = \delta^2/(8\times\card{\mathrm{OVWLine}(L_0)})$.
For each~$r \in L_1$, let (if it exists) $S_r\subseteq D$ by an OVW-line satisfying $\Lc(S)\subseteq L_0$
and $\dens_{A^{r-\size S_r}}(D^{S_r})>\h\delta$.
By \Cref{th:finitary-part-density-OVW}, there are at most $N_1$ many $r$ for which $S_r$ does not exist. Since $L_0$ is finite,
by the infinite pigeonhole principle, there is an OVW-line $S$ so that the set $L = \{r\in L_1: S_r = S\}$ is infinite.
The set $L$ is computable and $S$ and $L$ can both be found uniformly in any~$\emptyset'$-PA degree.
\end{proof}

We are now ready to prove \Cref{th:density-OVW}.

\begin{proof}[Proof of \Cref{th:density-OVW}]
It is similar to the proof of Theorem \ref{thm:piecewise-syndetic-ovw} mutatis mutandis,
using \Cref{lem_technical-density} instead of \Cref{lem:coloring-thick-left-ps} and positive upper density instead of piecewise syndeticity.
%
%
%
%
%
%
\end{proof}

As explained in \Cref{rem:full-induction-aca}, although the proofs of \Cref{th:finitary-part-density-OVW} and of \Cref{lem_technical-density} are over~$\RCA$, the proof of \Cref{th:density-OVW} holds in $\ACA_0$ with restricted induction.

%% file: parts/proof-ho-carlson-simpson-acap.tex
As mentioned in the introduction, many partition theorems admit higher-order counterparts. In the case of variable word theorems, an $n$-dimensional version consists of coloring $n$-variable words rather than words.
The two applications of the variable word theorems that we are going to study in \Cref{sect-cdrt-acap} and \Cref{sect_henson-graph} involve their higher-order counterparts.

There exists a simple inductive proof of the higher-order version of the Carlson-Simpson lemma.
The general idea consists considering the variables as part of the alphabet. Typically, the Carlson-Simpson lemma for dimension 1 and alphabet of size~$k$ uses $\omega$ applications the Carlson-Simpson lemma for dimension~0 and alphabet of size~$k+1$. The variables being treated as part of the alphabet, they will occur infinitely often in the solution. As a consequence, even using the Ordered Variable Word as the base statement, which is a stronger statement where each variable occurs finitely often, the resulting higher-order version yields only a solution to Carlson-Simpson's lemma.

The following inductive proof is standard in combinatorics. It essentially corresponds to the original proof of Carlson and Simpson~\cite[Section 2]{Carlson1984dual}. It was studied in the reverse mathematical setting by Dzhafarov, Flood, Solomon and Brown Westrick~\cite[Section 3.5]{Dzhafarov2017Effectiveness}. We include the proof of the sake of completeness.

\begin{theorem}\label[theorem]{thm:higher-order-csl-acap}
  For every $n \in \omega$, $\aca^{+}_0 \vdash \forall k \forall \ell \CSLD n k \ell$.
\end{theorem}

The proof uses induction on~$n$. 
   Let $A$ be an alphabet of size $k$ ; when we say $\omega$-variable word, it means over $A$ unless claimed otherwise. We will use the notion of \emph{prehomogeneous} $\omega$-variable word (see Definition \ref{def-prehomo}) to reduce an $\ell$-coloring of $A^{<\omega,n+1}$ to an $\ell$-coloring of $A^{<\omega,n}$.

    
  \begin{definition}\label{def-prehomo}
An $\omega$-variable word $W$ over $A$ is \emph{prehomogeneous} for a coloring $f : A^{<\omega,n+1} \to \ell$ if for every $s\in A^{<\omega,n}$ and $t_0, t_1\in A^{<\omega,n+1}$ such that $s^\frown x_n$ is prefix of both $t_0$ and $t_1$, $f(W[t_0])=f(W[t_1])$.
\end{definition}

For $\omega$-variable words $W,\h W$ and $m\in\omega$,
we write $\h W\leq_m  W$ iff $\h W = W[V]$
where $V $ is an $\omega$-variable word so that $z_0\cdots z_{m-1}\prec V$.
Clearly $\leq_m$ is transitive and $\leq_{m+1} $ implies $\leq_m$.
%
%

    \begin{lemma}[$\ACA_0$]\label{lem:one-step-hocsl-acap}
    Let $W$ be an $\omega$-variable word, and $s\in A^{<\omega,n}$. There exists an $\omega$-variable word $\h W\leq_{|s|+1}W$ and a color $i$ such that for every  
     $t\in A^{<\omega,n+1}$ with $s^\frown x_n\preceq t$, we have $f(\h W[t])=i$. 
  \end{lemma}
  \begin{proof}
    Let $\hat{A} = A \sqcup \{x_0, \dots, x_n\}$.
    We define $f_s : \hat{A}^{<\omega}\to \ell$ in the following way: for every $u \in \hat{A}^{<\omega}$, $f_s(u) = f(W[s^\frown x_n{}^\frown u])$.
    Note that $f_s$ can be seen as an instance of $\CSL {k+n+1} \ell$. Let $U$ be an $\omega$-variable word
     over $\h A$ with variable set $(y_n)_{n \in \NN}$ such that $f_s$ is constant, of value $i<\ell$ on $\{ U[u] : u \in \hat{A}^{<\omega} \}$.
    Replacing $(x_m:m\leq n)$ in $U$ by $(z_{j_m}:m\leq n)$ where $j_m$ is the first occurrence $x_m$
     in $s^\frown x_n$,
     we obtain an   $\omega$-variable word $\h U$ (over $A$ instead of over $\h A$);
     and
    let $\h W = W[z_0\cdots z_{|s|}^\frown \h U]$. 
      The motivation to use $\h U$ is that
     for every word $\h u$ over $\h A$, there is a word $u$ over $\h A$
     such that $\h W[s^\frown x_n^\frown \h u] = W[s^\frown x_n^\frown U[u]]$
     (this not necessarily true for other $\h s\in A^{<\omega,n}$).
   Therefore, 
   for every   $t\in A^{<\omega,n+1}$ with $s^\frown x_n\preceq t$,
    say $t= s^\frown x_n^\frown \h u $, we have, for some word $u$ over $\h A$,
    $$
    f(\h W[t]) = f(\h W[s^\frown x_n^\frown \h u]) = 
    f( W[s^\frown x_n^\frown U[u]]) = f_s(U[u])
      = i.
    $$
    Clearly $\h W\leq_{|s|+1} W$ and is an $\omega$-variable word (over $A$).
    Thus we are done.
  \end{proof}

We are now ready to prove \Cref{thm:higher-order-csl-acap}.

\begin{proof}[Proof of \Cref{thm:higher-order-csl-acap}]
The proof uses induction on $n$. The case $n = 0$ is a direct consequence of \Cref{th:density-OVW}, since an ordered $\omega$-variable word is a particular case of $\omega$-variable word. Assume $\forall k \forall \ell \CSLD n k \ell$ holds and fix $k$, $\ell$.
    We prove $\CSLD {n+1} k \ell$. 
    
Let $A$ be an alphabet of size $k$ and fix a coloring $f : A^{<\omega,n+1} \to \ell$.
  Iterating \Cref{lem:one-step-hocsl-acap}, there exists a sequence of $\omega$-variable words
   $W=W_0\geq_{n+1} W_1\geq_{n+2} W_2\geq_{n+3}\cdots$
    as well as a coloring on words $g:A^{<\omega,n}\to\ell$  such that
   for every $m\in\omega$, every $s\in A^{<\omega,n}$, every $t\in A^{<\omega,n+1}$ with 
   $s^\frown x_n\preceq t$ and $|s| = n+m$, we have
   $f(W_{m+1}[t]) = g(s)$;
   which implies $f(W_{\h m}[t]) = g(s)$
   for all $\h m\geq m+1$.
   It is easy to see that the limit $W_\omega= \lim_m W_m$ exists
   (by definition of $\leq_m$).
   Now for every $s\in A^{<\omega,n}$ and $t\in A^{<\omega,n+1}$ with
   $s^\frown x_n\preceq t$, we have (let $\h m$ be sufficiently large),
   $$f(W_\omega[t]) = f(W_{\h m}[t])
    = g(s). $$
    That is, $W_\omega$ is  prehomogeneous for $f$.
  It remains to restrain the space so that $g$ is constant. Let $W'$ be the $\omega$-variable word given by the application of $\CSLD n k \ell$ to the coloring $g$, so that $g\circ W'$ is constant for some color $i < \ell$.
  Let $t \in A^{<\omega, n+1}$ and $s \in A^{<\omega,n}$ be such that $s{}^\frown x_n\preceq t$. Then  
  $$f(W_\omega[W'[t]]) = g(W'[s]) = i.$$ Thus, $W=W_\omega[W']$ is a witness of the theorem. This completes the proof of \Cref{thm:higher-order-csl-acap}.  
\end{proof}

When constructing
the prehomogeneous variable word $W_\omega$,
we invoke $\msf{CSL}(k,\ell)$ infinitely many times which blows up
the complexity. However, it is unknown whether 
$\msf{CSL}(k,\ell)$ holds in $\RCA_0$.
Were it provable in $\RCA_0$, then $\CSLD n k\ell$
and results on the Dual Ramsey Theorem and the Big Ramsey degree 
of universal triangle-free graph
would be provable in $\ACA_0$. It is not even known whether $\RCA_0$ implies $\msf{CSL}(2,\ell)$.
\begin{question}
[Joe Miller and Solomon, \cite{Miller2004Effectiveness}]
Does $\RCA_0$ imply $\msf{CSL}(k,\ell)$?
\end{question}



%% file: parts/proof-cdrt-acap.tex
The first application of the Carlson-Simpson lemma is a dual version of Ramsey's theorem. This was actually the theorem which motivated the proof of the Carlson-Simpson theorem. 
While Ramsey's theorem is about finite colorings of the integers, the Dual Ramsey theorem is about finite colorings of colorings. 

For~$\alpha \in \omega^{+}$, we note $(\omega)^\alpha$ the class of partitions of $\omega$ in exactly $\alpha$ sets. Such a partition can be seen as a surjective ordered function from $\omega$ to $\alpha$. Therefore $(\omega)^\alpha$ inherits of a natural topology from $\alpha^\omega$. Given a partition $p \in (\omega)^\omega$ and $\alpha \in \omega^{+}$, we write $(X)^\omega$ for the class of $Y \in (\omega)^\alpha$ coarsening $X$, that is, for every~$n, m$ such that $X(m) = X(n)$, then $Y(m) = Y(n)$. Carlson and Simpson~\cite{Carlson1984dual} proved the following theorem:

\begin{theorem}[Borel Dual Ramsey theorem]
    For any $n,\ell \in \omega$, let $C_0 \cup ... \cup C_{\ell-1}$ be a partition of $(\omega)^n$ where each~$C_i$ is a Borel set. Then there is partition $X \in (\omega)^\omega$ such that $(X)^n \subseteq C_i$ for some~$i < \ell$.
\end{theorem}

The dual Ramsey theorem does not hold for arbitrary colorings. There exists in particular a 2-partition $C_0 \cup C_1$ of $(\omega)^2$ such that for all $X \in (\omega)^\omega$, neither $(X)^2 \subseteq C_0$ nor $(X)^2 \subseteq C_1$ (see \cite[Section 1.4]{Carlson1984dual}). On the other hand, the dual Ramsey theorem can be generalized to any coloring which admits the Baire property (see Pr\"omel and Voigt~\cite{proemel1985baire}).
From a mathematical viewpoint, it is well-known that Borel classes have the Baire property, hence the Baire version of the Dual Ramsey theorem implies its Borel version. However, from a computational and reverse mathematical viewpoint, the situation is more complicated. Indeed, the proof that Borel classes have the Baire property is non-trivial, and requires a careful analysis of the way to represent Borel classes in second-order arithmetic.

Dzhafarov, Flood, Solomon and Westrick~\cite{Dzhafarov2017Effectiveness} studied the reverse mathematics of the various versions of the Dual Ramsey theorem, namely, its restrictions to the Baire, the Borel, and the open colorings. They proved that the Borel Dual Ramsey theorem implies the Baire version, which is itself equivalent to the Open Dual Ramsey theorem. All these variants were proven to hold in~$\piooca_0$ by Slaman~\cite{Slaman1997note}. In their analysis, Dzhafarov and al.~\cite{Dzhafarov2017Effectiveness} reduced the Open Dual Ramsey theorem to a combinatorial statement, which they called the Combinatorial Dual Ramsey theorem:

\begin{theorem}[Combinatorial Dual Ramsey theorem]
    For any $n,\ell \geq 2$ let $C_0 \cup ... \cup C_{\ell-1}$ be a finite partition of $\emptyset^{<\omega, n-1}$, then there is a color $i < \ell$ and an $\omega$-variable word $W$ such that $\{W[u], u \in \emptyset^{<\omega, n-1}\} \subseteq C_i$
\end{theorem}

We will write $\ODRT{n}{\ell}$ the statement of the Open Dual Ramsey theorem for $\ell$-colorings of $(\omega)^n$ and  $\CDRT{n}{\ell}$ the statement of the Combinatorial Dual Ramsey theorem for $\ell$-colorings of $\emptyset^{<\omega, n-1}$.
Dzhafarov and al.~\cite{Dzhafarov2017Effectiveness} proved that $\ODRT{n}{\ell}$ and $\CDRT{n}{\ell}$ are equivalent over~$\RCA_0$. They proved $\CDRT{n}{\ell}$ from the Carlson-Simpson lemma inductively as in \Cref{sect-ho-csl-acap}, leaving the Carlson-Simpson lemma unproved, as a blackbox. The proof is optimal, in the following sense:

\begin{theorem}
	The following propositions are equivalent over~$\RCA_0$:
	\begin{enumerate}
	\item For all $n,k,\ell$, $\CSLD{n}{k}{\ell}$
	\item For all $n,\ell > 2$, $\CDRT{n}{\ell}$
	\end{enumerate}
\end{theorem}
\begin{proof}
	(1) clearly implies $(2)$ since $\CDRT{n}{\ell}$ is exactly $\CSLD{n-1}{0}{\ell}$. We just need to prove that $(2)$ implies (1). Let $n,k,\ell$ some integers, and we consider the statement $\CSLD{n}{k}{\ell}$. Let $A = \{a_0,...,a_{k-1}\}$ an alphabet of size $k$, and $C_0 \sqcup \dots \sqcup C_{\ell-1}$ a partition of $A^{<\omega,n}$. This induces a partition $\hat C_0 \sqcup \dots \sqcup \hat C_{\ell-1}$ of $\emptyset^{\omega,k+n}$ by replacing the first $k$ variables by a letter in $A$. Now, by using $\CDRT{k+n+1}{\ell}$, we have a color $i$ and an $\omega$-variable word $\hat W$ on alphabet $\emptyset$ such that $\{\hat W[u], u \in \emptyset^{<\omega, k+n}\} \subseteq \hat C_i$. Then consider the $\omega$-variable word $W = \hat W[a_0a_1...a_{k-1}x_0x_1...]$ on alphabet $A$. For all $u \in A^{<\omega,n}$, since the letters in $a_0a_1...a_{k-1}^\smallfrown u$ appear in order, this word can be seen as a ${k+n}$-variable word $\hat u \in \emptyset^{\omega,k+n}$. Therefore, since $\hat W[\hat u] \in \hat C_i$, we also have $\hat W[a_0a_1...a_{k-1}^\smallfrown u] = W[u] \in C_i$, meaning that $W$ is a solution to our instance of $\CSLD{n}{k}{\ell}$.
\end{proof}

Thanks to our new analysis of the Carlson-Simpson lemma, we prove that the Combinatorial Dual Ramsey theorem holds in~$\ACA^{+}_0$. It follows that the open and Baire versions of the Dual Ramsey theorem also hold in~$\ACA^{+}_0$, which is a dramatical improvement from the previous bound of $\piooca_0$.

\begin{corollary}
	For all $n \geq 2$, $\aca^{+}_0 \vdash \forall \ell \ODRT{n}{\ell}$
\end{corollary}

%% file: parts/big-ramsey-triangle-free.tex
Our second application comes from Structural Ramsey Theory.
We consider graphs as relational structures and denote by $\mathbb H_3$ the universal triangle-free graph, or the triangle-free Henson graph. Given integers $k, \ell \in \omega$, structures $\mathbf A, \mathbf B$ and $\mathbf C$ we write $\mathbf C \longrightarrow (\mathbf A)^{\mathbf B}_{k,\ell}$ for the following statement:

\begin{definition}
$\mathbf C \longrightarrow (\mathbf A)^{\mathbf B}_{k,\ell}$: For any partition $C_0 \cup ... \cup C_{k-1}$ of the embeddings of $\mathbf B$ in $\mathbf A$ there is an embedding $f: \mathbf C \mapsto \mathbf A$ such that all the embeddings of $\mathbf B$ in $f(\mathbf C)$ intersect at most $\ell$ many $C_i$'s.
\end{definition}

The \textit{big Ramsey degree} of $\mathbf B$ in $\mathbf A$ is the smallest $L \in \omega^{+}$ such that $\mathbf A \longrightarrow (\mathbf A)^{\mathbf B}_{k,L}$ for every $k \in \omega$, and a infinite structure $\mathbf A$ is said to have \textit{finite big Ramsey degrees} if for all finite substructure of $\mathbf B$, $\mathbf B$ has a finite big Ramsey degree in $\mathbf A$.
The study of big Ramsey degrees was initiated by Laver, who proved the order of rationals has finite big Ramsey degrees~\cite{devlin1979rationnals}.

Most recent results on big Ramsey degrees rely on a combinatorial core which is usually a tree partition theorem. The most famous one is Milliken's tree theorem, which was used by Devlin~\cite{devlin1979rationnals} in 1979 to give a precise characterization of big Ramsey degrees on the order of rationals, and by Laflamme, Sauer, and Vuksanovic~\cite{Sauer2006rado} in 2006 to characterize the big Ramsey degrees of the Rado graph. The reverse mathematics of Milliken's tree theorem and its applications were studied by Anglès d'Auriac, Cholak, Dzhafarov, Monin and Patey~\cite{2020millikens}, who proved that they all hold in~$\ACA_0$.

However, Milliken's tree theorem fails to serve as a combinatorial core to prove that some non-universal relational structures admit finite big Ramsey degrees.
In 2020, Dobrinen~\cite{dobrinen2020ramsey} proved that the triangle-free Henson graph admits finite big Ramsey degrees by proving a combinatorial statement about coding trees, using an involved notion of forcing.
More recently, Hubička~\cite{hubicka2020big} gave an alternative proof of Dobrinen's theorem using the higher-order Carlson-Simpson lemma as a combinatorial core. Although this proof is less accurate than the original proof of Dobrinen, it has the advantage of relying on a partition theorem with a known combinatorial proof. In this section, we analyse the proof by Hubička and show that it holds in~$\ACA^+_0$. The analysis is straightforward, but we include it for the sake of completeness.

By the usual back-and-forth argument, any computable copy of the triangle-free Henson graph is computably isomorphic. We will first enrich the set $\{0\}^{<\omega, 1}$ of variable words over the unary alphabet~$\{0\}$ with a symmetric irreflexive binary relation $E$ so that $(\{0\}^{<\omega, 1}, E)$ is a universal triangle-free graph, hence computably isomorphic to~$\mathbb H_3$.


\begin{definition}
  Let $E$ be the symmetric binary relation on $\{0\}^{<\omega, 1}$ defined as follows: for $v,w\in \{0\}^{<\omega, 1}$, $v E w$ if and only if $|v|\neq|w|$ and, assuming $|v|<|w|$:
  \begin{enumerate}
  \item Passing number property: $w(|v|)=x_0$
  \item Triangle-freeness condition: there is no $i<|v|$ with $v(i)=w(i)=x_0$.
  \end{enumerate}
	We write $\mathbb G$ the graph $(\{0\}^{<\omega, 1}, E)$
	\end{definition}

One can see the set~$\{0\}^{<\omega, 1}$ as the full binary tree $2^{<\omega}$ truncated from the set $\{0, 00, 000, \dots\}$. Based on this intuition, the previous construction is very similar to the one of Laflamme, Sauer, and Vuksanovic~\cite{Sauer2006rado} where they represent a Rado graph by the full binary tree with an edge relation based on the passing number property. The following lemma shows how the second property of the edge relation ensures that $\mathbb G$ is triangle-free.

\begin{lemma}
	$\mathbb G$ is triangle-free.
	\end{lemma}
\begin{proof}
	Assume by contradiction that we have three words $s,t,u \in \{0\}^{<\omega, 1}$, forming a triangle. Without loss of generality, we assume $|s| < |t| < |u|$. Therefore we must have $t(|s|) = x_0$. But we also have $u(|s|) = x_0$, and this contradicts the triangle-freeness condition between $t$ and $u$.
\end{proof}

The triangle-freeness condition ensures that there will not be too many edges, the risk being that the resulting edge relation is too restrictive and that we lose universality. The following theorem shows that the edge relation is general enough.
	
\begin{theorem}[$\RCA_0$]
\label{csth0}
	$\mathbb G$ is the universal triangle-free graph.
	\end{theorem}
\begin{proof}
	We define $\varphi : \mathbb H_3 \mapsto \{0\}^{<\omega, 1}$	 with for all $i, \varphi(i) = s$ where $s$ is such that $|s| = i$ and for all $j<i$ :
	\begin{itemize}
  \item $s(j) = x_0$ if there is an edge between $i$ and $j$ in $\mathbb H$
	\item $s(j) = 0$ otherwise
  \end{itemize}
	It is clear that $\varphi$ is an embedding $\varphi : \mathbb H_3 \mapsto \mathbb G$ therefore since $\mathbb G$ is triangle-free and by unicity of the universal triangle-graph, $\mathbb G$ is the universal triangle-free graph.
\end{proof}

The following observation shows that for every $\omega$-variable word~$W$,
the induced subgraph $(\{W[u] : u \in  \{0\}^{<\omega, 1} \}, E)$ remains universal:
\begin{lemma}[$\RCA_0$]
\label[lemma]{lem:subspace-still-universal}
Let~$W$ be an $\omega$-variable word. Then for every~$u, v \in \{0\}^{<\omega, 1}$, $u E v$ iff $W[u] E W[v]$.	
In particular, $(\{W[u] : u \in  \{0\}^{<\omega, 1} \}, E)$ is the universal triangle-free graph.
\end{lemma}
\begin{proof}
Say~$|u| < |v|$. Then $|W[u]| < |W[v]|$. The passing number property is preserved: $v(|u|) = x_0$ iff $W[v](|W[u]|) = x_0$.  Let us show that the triangle-freeness condition is preserved. Suppose there is some~$i < |u|$ such that $u(i) = v(i) = x_0$. Then let~$j$ be the position of any occurrence of~$x_i$ in~$W$. We have $W[u](j) = u(i) = v(i) = W[v](j) = x_0$. Conversely, suppose there is some~$j < |W[u]|$ such that $W[u](j) = W[v](j) = x_0$. Then $W(j) = x_i$ for some~$i < |u|$. In particular, $W[u](j) = u(i)$ and $W[v](j) = v(i)$, so $u(i) = v(i) = x_0$.
\end{proof}

\begin{remark}
The previous lemma is precisely the reason why we consider variable words over~$\{0\}$ rather than binary words, and use the Carlson-Simpson lemma rather than Milliken's tree theorem. Indeed, identifying $\{0\}^{<\omega, 1}$ as $2^{<\omega} \setminus \{0, 00, 000, \dots\}$, an application of Milliken's tree theorem yields a strong infinite subtree~$S$. The issue that the graph $(S, E)$ is not universal in general. For example, if there is some~$n \in \omega$ such that every~$w \in S$ satisfies $w(n) = 1$, then the graph~$(S, E)$ is an anticlique, by the triangle-freeness condition.
\end{remark}
	
The following notions of envelope and embedding types are now standard in the study of big Ramsey degrees. We define the appropriate notions under the scope of variable words :

\begin{definition}
  Let $A$ an alphabet and $S$ a set of finite variable words over $A$. An \emph{envelope} of $S$ is a $(<\omega)$-variable word $W$ such that for all $s \in S$ there is a variable word $t$ such that $W[t] = s$. We say that $W$ is \emph{minimal} if there is no envelope of $S$ with fewer variables than $W$.
	\end{definition}
	
A simple computation shows that if~$W$ is a minimal envelope of~$S$, then
\begin{align}\label{cseq0}
\text{it has at most $2^{|S|}+|S|-1$ variables (see~\cite[Proposition 3.1]{hubicka2020big}).}
\end{align}
 In what follows, given two graphs~$F, G$, we write ${G \choose F}$ for the set of all embeddings from~$F$ to~$G$.
We write $F\cong_\varphi G$ iff $F$ is isomorphic to $G$ via $\varphi$.	
\begin{theorem}[$\RCA_0 + \forall k \CSLD {2^n+n-1} 1 k$]\label[theorem]{thm:henson-csl}
	Let $F$ be a finite triangle-free graph of size~$n$, there is an integer $\ell$ such that for any integer $k > 0$ and any finite coloring $\chi : \binom{\mathbb G}{F} \mapsto k$, there exists $f \in \binom{\mathbb G}{\mathbb G}$ such that $\chi$ attain at most $\ell$ colors on $\binom{f(\mathbb G)}{F}$.
	\end{theorem}
	
\begin{proof}
Let $\h A$ be $\{0\}\cup\{x_0\}$.
Firstly, we note that $\chi$ give rise to a coloring $\h\chi\leq_T \chi$
of $\{0\}^{<\omega,2^n+n-1}$, where $\h\chi(u)$ records the coloring profile
of $\langle \varphi\ :\ T\subseteq \h A^{\leq 2^n+n-1}\text{ and }F\cong_\varphi u[T] \rangle$,
namely $$\h\chi(u) = \langle\chi(\varphi )\ :\ T\subseteq \h A^{\leq 2^n+n-1}\text{ and }F\cong_{\varphi} u[T]\rangle.$$
By $\msf{CSL}^{2^n+n-1}( 1, \h{k})$ (for some sufficiently large $\h k$),
there exists
an $\omega$-variable word $W$ (over $\{0\}$)
such that $\{W[u]: u\in \{0\}^{<\omega,2^n+n-1}\}$ is homogeneous for $\h\chi$.
By Lemma \ref{lem:subspace-still-universal}, the graph $\{W[t]: t\in\h A^{<\omega}\}$ is isomorphic to~$\mathbb G$.
It suffices to show that for any embedding
$\varphi$ of $F$ into the graph $\{W[t]: t\in\h A^{<\omega}\}$, we have
$\chi(\varphi)\in L$ where $L$ is the range of $\chi$, namely for every $u\in \{0\}^{<\omega,2^n+n-1} $,
let $\h u= W[u]$,
$$L=\{\chi(\varphi): T\subseteq  \h A^{\leq 2^n+n-1}\text{ and }F\cong_{\varphi} \h u[T]\}.$$
Let $S\subseteq \h A^{<\omega}$ so that
$W[S] = \varphi(F)$; which means  $|S| = n$.
By (\ref{cseq0}), there is a $u\in \{0\}^{<\omega,2^n+n-1}$ and $T\subseteq \h A^{\leq 2^n+n-1}$
such that $u[T] = S$.
Let $\h u = W[u]\in \{0\}^{<\omega,2^n+n-1}$.
Clearly, $F\cong_\varphi \h u[T]$. So $$\chi(\varphi)\in
\{\chi(\varphi): T\subseteq \h A^{\leq 2^n+n-1}\text{ and }F\cong_{\varphi } \h u[T]\}=L.$$
The last equality follows by homogeneity of $\h \chi$ on $\{W[u]: u\in \{0\}^{<\omega,2^n+n-1}\}\ni\h u$.
Thus we are done.
\end{proof}

\begin{corollary}[$\ACA^{+}_0$]
The triangle-free Henson graph admits finite big Ramsey degrees.
\end{corollary}
\begin{proof}
Immediate by \Cref{thm:henson-csl} and \Cref{thm:higher-order-csl-acap}.
\end{proof}
	